\pdfoutput=1
\documentclass[a4paper,11pt]{article}

\usepackage[T1]{fontenc}
\usepackage[utf8]{inputenc}
\usepackage[margin=1in]{geometry}
\usepackage{lmodern}
\usepackage{amsmath,amsthm,amssymb,amsfonts}
\usepackage{enumitem}
\usepackage{booktabs}
\usepackage{float}
\usepackage{cite}
\usepackage{microtype}
\usepackage[hidelinks]{hyperref}

\hypersetup{
  pdftitle={Pairwise Separation and Compactness in Soft Bitopological Spaces via Soft Elements},
  pdfauthor={Subhasis Ray and Titu Barik},
  pdfsubject={General Topology},
  pdfkeywords={soft set, soft element, soft topology, bitopological space, separation axiom, soft compactness}
}

\numberwithin{equation}{section}
\theoremstyle{definition}
\newtheorem{lemma}{Lemma}[section]
\newtheorem{proposition}[lemma]{Proposition}
\newtheorem{example}[lemma]{Example}
\newtheorem{theorem}[lemma]{Theorem}
\newtheorem{corollary}[lemma]{Corollary}
\newtheorem{definition}[lemma]{Definition}
\newtheorem{remark}[lemma]{Remark}

\newcommand{\Pow}{\mathcal{P}}
\newcommand{\SE}{\operatorname{SE}}
\newcommand{\softin}{\mathrel{\in_s}}
\newcommand{\softnotin}{\mathrel{\notin_s}}
\newcommand{\softsub}{\mathrel{\subseteq_s}}
\newcommand{\softcup}{\mathbin{\cup_s}}
\newcommand{\softcap}{\mathbin{\cap_s}}
\newcommand{\Can}{\mathsf{Can}}
\newcommand{\Cyl}{\operatorname{Cyl}}

\title{Pairwise Separation and Compactness in Soft Bitopological Spaces via Soft Elements}

\author{
Subhasis Ray\thanks{Corresponding author. Department of Mathematics, Visva-Bharati University, Santiniketan, West Bengal, India. Email: \href{mailto:subhasis.ray@visva-bharati.ac.in}{\nolinkurl{subhasis.ray@visva-bharati.ac.in}}. ORCID: \href{https://orcid.org/0000-0003-1100-4632}{0000-0003-1100-4632}.}
\and
Titu Barik\thanks{Ph.D. Scholar, Department of Mathematics, Visva-Bharati University, Santiniketan, West Bengal, India. Email: \href{mailto:titubarik96@gmail.com}{\nolinkurl{titubarik96@gmail.com}}. ORCID: \href{https://orcid.org/0009-0003-4317-9979}{0009-0003-4317-9979}.}
}

\date{}

\begin{document}
\maketitle

\begin{abstract}
Let $F\colon A\to\mathcal P(X)$ be a soft set, and let $\SE(F)$ be
the set of all maps $a\colon A\to X$ such that
$a(t)\in F(t)$, for every $t\in A$. If $\tau$ is a soft topology on
$F$, the sets $\SE(H)$, where $H\in\tau$, form a basis for an ordinary
topology $\tau^{\mathrm e}$ on $\SE(F)$, called the \emph{induced
soft-element topology}. Hence, two soft topologies on $F$ induce a
bitopological space on $\SE(F)$. We define pairwise soft $T_0$,
$T_1$, and $T_2$ axioms using soft elements and prove that they are
equivalent to the corresponding pairwise axioms of the induced
soft-element bitopological space. For canonical soft bitopologies,
these properties are also equivalent to the same axioms in every
component bitopological space. We also study pairwise soft
compactness. For canonical soft bitopological spaces, pairwise soft
compactness implies pairwise compactness of every component space,
and the converse holds when the parameter set is finite. An example
with infinitely many parameters shows that the finiteness condition
is needed.
\end{abstract}



\bigskip

% ====================================================================
\section{Introduction}

Molodtsov \cite{Molodtsov1999} introduced a soft set as a parameterized family of subsets of a universal set. Maji et al. \cite{Maji2003} developed the early theory, while Pei and Miao \cite{PeiMiao2005}, Ali et al. \cite{Ali2009}, and \c{C}a\u{g}man and Engino\u{g}lu \cite{CagmanEnginoglu2010} examined or modified the basic operations. Since several conventions occur in the literature, the present paper fixes one parameter set and uses pointwise soft inclusion, union, and intersection throughout.

Two forms of soft topology appeared in 2011.
In the approach of \c{C}a\u{g}man et al. \cite{Cagman2011},
the soft open sets are soft subsets of a fixed soft set $F$, and
the sets $F(t)$ may vary with $t$.
Shabir and Naz \cite{ShabirNaz2011} used the absolute soft set
$\widetilde{X}$, where
\(
\widetilde{X}(t)=X, \; \text{for every } t\in A.
\)
Thus, the same set $X$ is used at every parameter.
We use the fixed-soft-set approach. Further studies of separation
and compactness in soft topological spaces include
\cite{Zorlutuna2012,Peyghan2014,HussainAhmad2015,
BayramovAras2018}.

Ray and Goldar \cite{RayGoldar2017} studied soft sets through their
soft elements. A soft element of $F$ is a function $a\colon A\to X$
such that
\(
a(t)\in F(t), \; \text{for every } t\in A.
\)
Thus
\[
\SE(F)=\prod_{t\in A}F(t).
\]
Chiney and Samanta \cite{ChineySamanta2019} gave a new definition
of soft topology using elementary union and elementary intersection.
Goldar and Ray \cite{GoldarRay2019} used soft elements to study
separation axioms and compactness. Nazmul and Samanta
\cite{NazmulSamanta2013} studied neighborhood properties in soft
topological spaces. Related product and soft-element constructions
appear in Matejdes \cite{Matejdes2016,Matejdes2021},
Ta\c{s}k\"opr\"u and Alt\i nta\c{s} \cite{TaskopruAltintas2021},
and Alt\i nta\c{s} et al. \cite{AltintasTaskopruSelvi2021}.
Azzam et al. \cite{Azzam2022} also reconsidered the construction
of soft topologies and their operators. These papers provide the
background for treating $\SE(F)$ as an ordinary set of points.

Kelly \cite{Kelly1963} introduced bitopological spaces and the
pairwise Hausdorff axiom. Fletcher et al.
\cite{FletcherHoylePatty1969} gave the pairwise $T_0$ condition
and defined pairwise compactness using pairwise open covers.
Reilly \cite{Reilly1972} later studied pairwise $T_0$ and $T_1$
separation properties, while Cooke and Reilly
\cite{CookeReilly1975} further studied bitopological compactness.

Soft bitopological spaces were studied by \c{S}enel and \c{C}a\u{g}man
\cite{SenelCagman2014}, Ittanagi \cite{Ittanagi2014}, Naz et al.
\cite{NazShabirAli2015}, and Acharjee and Tripathy
\cite{AcharjeeTripathy2017}. Nordo \cite{Nordo2022} considered spaces
with several soft topologies, while Askandar et al. \cite{Askandar2022}
studied further classes of open sets in soft bitopological spaces.

Most of these works define bitopological notions directly in terms
of soft sets or soft points. To our knowledge, they do not study how
two soft topologies on a fixed soft set $F$ induce an ordinary
bitopology on $\SE(F)$, whether pairwise soft $T_0$, $T_1$, and $T_2$
agree with the corresponding separation axioms of the induced space,
or how pairwise soft compactness is related to compactness of the
component spaces.

First we define the topology on $\SE(F)$. For a soft open set $H$,
let $\SE(H)$ be the set of soft elements belonging to $H$. The family
\(
\{\SE(H):H\in\tau\}
\)
need not be closed under unions, but it covers $\SE(F)$ and is closed
under finite intersections. It is therefore a basis for a topology
$\tau^{\mathrm e}$ on $\SE(F)$, called the \emph{induced
soft-element topology}. Two soft topologies $\tau_1$ and $\tau_2$
on $F$ then give the ordinary bitopological space
\(
\bigl(\SE(F),\tau_1^{\mathrm e},\tau_2^{\mathrm e}\bigr).
\)

We prove that pairwise soft $T_j$ is equivalent to pairwise $T_j$
in the induced soft-element bitopological space, for $j=0,1,2$.
When both soft topologies are canonical, these conditions are also
equivalent to the corresponding pairwise axioms in every component
space. In the same canonical setting, pairwise soft compactness
implies pairwise compactness of every component space. If the
parameter set is finite, the converse holds even when the soft
topologies are not canonical. An example shows that the converse
may fail when the parameter set is infinite.

Section~\ref{sec:prelim} fixes the definitions and notation. Section~\ref{sec:element-topology} constructs the induced soft-element topology. Sections~\ref{sec:separation} and~\ref{sec:compactness} deal with pairwise separation and compactness, respectively. The last section gives a brief conclusion.

% ====================================================================
\section{Preliminaries and notation}\label{sec:prelim}

\subsection{Soft sets and soft elements}

Throughout the paper, $X$ is a nonempty universal set, $A$ is a nonempty parameter set, and
\[
F\colon A\longrightarrow \Pow(X)
\]
is a fixed soft set. We assume that $F(t)\neq\varnothing$, for every $t\in A$, and that $\SE(F)\neq\varnothing$. The second assumption is stated separately because, for an arbitrary infinite parameter set, it may require a choice principle.

\begin{definition}[Soft set and basic operations \cite{Molodtsov1999,PeiMiao2005}]\label{def:soft-set}
A \emph{soft set over $X$ with parameter set $A$} is a map $H\colon A\to\Pow(X)$. For soft sets $H$ and $K$ with the same parameter set, write
\[
H\softsub K
\quad\Longleftrightarrow\quad
H(t)\subseteq K(t),\quad\text{for every }t\in A.
\]
The \emph{null soft set} $\Phi$ is given by $\Phi(t)=\varnothing$, for every $t\in A$. The \emph{soft union} and \emph{soft intersection} of $H$ and $K$ are defined by
\[
(H\softcup K)(t)=H(t)\cup K(t),
\qquad
(H\softcap K)(t)=H(t)\cap K(t),
\qquad\text{for every }t\in A.
\]
Arbitrary soft unions are defined in the same pointwise way.
\end{definition}

The operations in Definition~\ref{def:soft-set} are the ones used in every later statement. In particular, all soft sets in the paper have the fixed parameter set $A$.

\begin{definition}[Soft element \cite{RayGoldar2017}]\label{def:soft-element}
A \emph{soft element} of $F$ is a map $a\colon A\to X$ such that
\[
a(t)\in F(t),\qquad\text{for every }t\in A.
\]
The set of all soft elements of $F$ is denoted by $\SE(F)$.

For a soft subset $H\softsub F$, define
\[
\SE(H)=\{a\in\SE(F):a(t)\in H(t),\text{ for every }t\in A\}.
\]
We write $a\softin H$ when $a\in\SE(H)$, and $a\softnotin H$ when $a\notin\SE(H)$.
\end{definition}

\begin{example}\label{ex:soft-elements}
Let $A=\{p,q\}$, $F(p)=\{x,y\}$, and $F(q)=\{u,v\}$. A soft element is a function on $A$. Its graph has the form
\[
\{(p,a(p)),(q,a(q))\}.
\]
After fixing the order $(p,q)$, we abbreviate this function by the ordered pair $(a(p),a(q))$. Thus,
\[
\SE(F)=\{(x,u),(x,v),(y,u),(y,v)\}.
\]
If $H(p)=\{x\}$ and $H(q)=\{v\}$, then $(x,v)\softin H$, while the other three soft elements do not belong to $H$.
\end{example}

\begin{remark}\label{rem:soft-points}
A soft element in Definition~\ref{def:soft-element} chooses a value at every parameter. It is different from a soft point, which is supported at one selected parameter; see Zorlutuna et al. \cite{Zorlutuna2012}. Soft elements are suited to the product $\prod_{t\in A}F(t)$ and will be used throughout this paper.
\end{remark}

\subsection{Soft topologies and canonical soft topologies}

\begin{definition}[Soft topology \cite{Cagman2011}]\label{def:soft-topology}
A family $\tau$ of soft subsets of $F$ is a \emph{soft topology on $F$} if
\begin{enumerate}[label=(\roman*)]
\item $\Phi,F\in\tau$;
\item the soft union of any family of members of $\tau$ belongs to $\tau$;
\item the soft intersection of any finite family of members of $\tau$ belongs to $\tau$.
\end{enumerate}
The pair $(F,\tau)$ is a soft topological space, and the members of $\tau$ are soft open sets.
\end{definition}

This is the fixed-soft-set approach of \c{C}a\u{g}man et al. \cite{Cagman2011}. The Shabir--Naz setting \cite{ShabirNaz2011} is obtained by taking $F(t)=X$, for every $t\in A$.

For $t\in A$, put
\begin{equation}\label{eq:component-topology}
\tau_t=\{H(t):H\in\tau\}.
\end{equation}

\begin{proposition}\label{prop:component-topology}
The family $\tau_t$ is a topology on $F(t)$, for every $t\in A$.  We call $\tau_t$ the
\emph{component topology of $\tau$ at $t$}.
\end{proposition}

\begin{proof}
Since $\Phi,F\in\tau$, we have $\varnothing=\Phi(t)\in\tau_t$ and $F(t)\in\tau_t$. Let $\{U_i:i\in I\}\subseteq\tau_t$ and set
\[
\mathcal H=\{H\in\tau:H(t)=U_i\text{ for some }i\in I\}.
\]
Let $G$ be the soft union of the family $\mathcal H$. Then, $G\in\tau$, and
\[
G(t)=\bigcup_{H\in\mathcal H}H(t)=\bigcup_{i\in I}U_i.
\]
Thus, arbitrary unions of members of $\tau_t$ belong to $\tau_t$ without choosing one representative for every index. For $U_1,\ldots,U_n\in\tau_t$, choose $H_k\in\tau$ with $H_k(t)=U_k$, for $k=1,\ldots,n$, and take their soft intersection. Hence, $\tau_t$ is closed under finite intersections and is a topology on $F(t)$.
\end{proof}

Ta\c{s}k\"opr\"u and Alt\i nta\c{s}
\cite{TaskopruAltintas2021} gave a related coordinate construction
for an $\varepsilon$-soft topology on the absolute soft set, using
elementary soft operations. Here we work on a fixed soft set $F$
and use the usual pointwise soft operations. We call the following
construction a \emph{canonical soft topology}.

\begin{definition}[Canonical soft topology]\label{def:canonical}
Let $\sigma_t$ be a topology on $F(t)$, for every $t\in A$. Define
\[
\Can(\{\sigma_t\}_{t\in A})
 =\{H\softsub F:H(t)\in\sigma_t,\text{ for every }t\in A\}.
\]
A soft topology $\tau$ is called \emph{canonical} if
$\tau=\Can(\{\tau_t\}_{t\in A}).$
\end{definition}

\begin{proposition}\label{prop:canonical-topology}
The family $\Can(\{\sigma_t\}_{t\in A})$ is a soft topology on $F$, and its component topology at $t$ is $\sigma_t$.
\end{proposition}

\begin{proof}
The null soft set and $F$ belong to the family because $\varnothing,F(t)\in\sigma_t$, for every $t\in A$. If $\{H_i:i\in I\}$ is a family of its members, let $H$ be its soft union. Then, $H(t)=\bigcup_{i\in I}H_i(t)\in\sigma_t$, for every $t\in A$, so $H$ belongs to the family. Finite soft intersections are treated in the same way. Hence, the family is a soft topology.

Its component topology is contained in $\sigma_t$ by definition. Conversely, if $U\in\sigma_t$, define $H(t)=U$ and $H(s)=F(s)$ when $s\neq t$. Then, $H\in\Can(\{\sigma_s\}_{s\in A})$ and $H(t)=U$. Thus, the component topology is exactly $\sigma_t$.
\end{proof}

\subsection{Classical bitopological notions}

Kelly \cite{Kelly1963} defined a bitopological space as a set with
two topologies. The pairwise $T_0$ condition below is due to
Fletcher et al. \cite{FletcherHoylePatty1969}. We use the pairwise
$T_1$ condition of Reilly \cite{Reilly1972}. Kelly called the
pairwise $T_2$ condition \emph{pairwise Hausdorff}.

\begin{definition}[Bitopological separation axioms]
\label{def:classical-pairwise}
A \emph{bitopological space} is a triple
$(Y,\sigma_1,\sigma_2)$, where $\sigma_1$ and $\sigma_2$ are
topologies on $Y$. For every two distinct points $x,y\in Y$:
\begin{enumerate}[label=(\roman*)]
\item the space is \emph{pairwise $T_0$} if either there is
      $U\in\sigma_1$ such that
      \[
      x\in U \quad\text{and}\quad y\notin U,
      \]
      or there is $V\in\sigma_2$ such that
      \[
      y\in V \quad\text{and}\quad x\notin V;
      \]

\item the space is \emph{pairwise $T_1$} if there are
      $U\in\sigma_1$ and $V\in\sigma_2$ such that
      \[
      x\in U,\qquad y\notin U,\qquad
      y\in V,\qquad x\notin V;
      \]

\item the space is \emph{pairwise $T_2$} if there are
      $U\in\sigma_1$ and $V\in\sigma_2$ such that
      \[
      x\in U,\qquad y\in V,\qquad U\cap V=\varnothing.
      \]
\end{enumerate}
\end{definition}

Fletcher et al. \cite{FletcherHoylePatty1969} call a cover
$\mathcal V$ of $Y$ pairwise open if
\[
\mathcal V\subseteq\sigma_1\cup\sigma_2
\]
and both $\mathcal V\cap\sigma_1$ and
$\mathcal V\cap\sigma_2$ contain a nonempty member. They call
$Y$ pairwise compact if every such cover has a finite subcover.
Reilly \cite{ReillyLocal1972} used the same convention in his
study of bitopological local compactness.

In this paper, we use a slightly different convention. We allow
covers of a subset $B\subseteq Y$, and we do not require the cover
to contain a nonempty member from both topologies.

\begin{definition}[Pairwise compactness]
\label{def:classical-compact}
Let $(Y,\sigma_1,\sigma_2)$ be a bitopological space and let
$B\subseteq Y$. A \emph{pairwise open cover} of $B$, in this paper,
is a family
\[
\mathcal V\subseteq\sigma_1\cup\sigma_2
\]
such that
\[
B\subseteq\bigcup_{V\in\mathcal V}V.
\]
The set $B$ is \emph{pairwise compact} if every pairwise open
cover of $B$ has a finite subfamily that covers $B$.
\end{definition}

Thus, a pairwise open cover in this paper may consist only of sets
from one of the two topologies.
\begin{table}[H]
\centering
\caption{Notations used in the paper.}\label{tab:notation}
\begin{tabular}{@{}p{3.8cm}p{11.4cm}@{}}
\toprule
Symbol & Meaning \\
\midrule
$A$, $X$ & parameter set and universal set \\
$F\colon A\to\Pow(X)$ & fixed soft set with $F(t)\neq\varnothing$, for every $t\in A$ \\
$H\softsub F$ & $H(t)\subseteq F(t)$, for every $t\in A$ \\
$\Phi$ & null soft set, $\Phi(t)=\varnothing$, for every $t\in A$ \\
$\softcup$, $\softcap$ & soft union and soft intersection \\
$\SE(F)$ & set of soft elements of $F$ \\
$a\softin H$ & $a(t)\in H(t)$, for every $t\in A$ \\
$\SE(H)$ & set of soft elements of $H$, viewed as a subset of $\SE(F)$ \\
$\tau_t$ & component topology defined in \eqref{eq:component-topology} \\
$\tau^{\mathrm e}$ & induced soft-element topology defined in Theorem~\ref{thm:e-topology} \\
$\Cyl_{t_0}(U)$
&
the soft set equal to $U$ at $t_0$ and to $F(t)$ for $t\neq t_0$
\\
\bottomrule
\end{tabular}
\end{table}

% ====================================================================
\section{The induced soft-element topology}\label{sec:element-topology}

\subsection{Basic sets}

For a soft subset $H\softsub F$,
\[
\SE(H)=\prod_{t\in A}H(t)
\]
as a subset of $\SE(F)=\prod_{t\in A}F(t)$. The set $\SE(H)$ may be empty.

\begin{lemma}\label{lem:se-identities}
For soft subsets $H,K\softsub F$,
\begin{align*}
\SE(F)&=\prod_{t\in A}F(t), & \SE(\Phi)&=\varnothing,\\
\SE(H\softcap K)&=\SE(H)\cap\SE(K), &
\SE(H)\cup\SE(K)&\subseteq\SE(H\softcup K).
\end{align*}
\end{lemma}

\begin{proof}
The first identity is Definition~\ref{def:soft-element}. Since $A$ is nonempty and $\Phi(t)=\varnothing$, for every $t\in A$, no soft element can belong to $\Phi$. Hence, $\SE(\Phi)=\varnothing$.

Let $a\in\SE(H\softcap K)$. Then,
\[
a(t)\in H(t)\cap K(t),\qquad\text{for every }t\in A,
\]
so $a\in\SE(H)$ and $a\in\SE(K)$. This proves
$\SE(H\softcap K)\subseteq\SE(H)\cap\SE(K)$. Conversely, let $a\in\SE(H)\cap\SE(K)$. Then, $a(t)\in H(t)$ and $a(t)\in K(t)$, for every $t\in A$. Hence, $a(t)\in(H\softcap K)(t)$, for every $t\in A$, and therefore $a\in\SE(H\softcap K)$. Thus,
\[
\SE(H\softcap K)=\SE(H)\cap\SE(K).
\]

Finally, let $a\in\SE(H)\cup\SE(K)$. If $a\in\SE(H)$, then $a(t)\in H(t)\subseteq H(t)\cup K(t)$, for every $t\in A$. The same conclusion holds when $a\in\SE(K)$. Therefore,
\[
a\in\SE(H\softcup K),
\]
and the stated inclusion follows.
\end{proof}

\begin{example}[The union inclusion can be strict]\label{ex:strict-union}
Let $A=\{1,2\}$, $F(1)=F(2)=\{0,1\}$, and define
\[
H_0(t)=\{0\},\qquad H_1(t)=\{1\}
\]
for both parameters. Then,
\[
\SE(H_0)\cup\SE(H_1)=\{(0,0),(1,1)\},
\]
whereas
\[
\SE(H_0\softcup H_1)=\SE(F)
 =\{(0,0),(0,1),(1,0),(1,1)\}.
\]
\end{example}

\begin{theorem}\label{thm:e-topology}
Let $(F,\tau)$ be a soft topological space. The family
\[
\mathcal B_{\tau}=\{\SE(H):H\in\tau\}
\]
is a basis for a topology on $\SE(F)$. We call the generated topology the \emph{induced soft-element topology} and denote it by $\tau^{\mathrm e}$. Thus,
\begin{equation}\label{eq:e-topology}
\tau^{\mathrm e}
 =\left\{\bigcup_{H\in\mathcal H}\SE(H):\mathcal H\subseteq\tau\right\}.
\end{equation}
\end{theorem}

\begin{proof}
Since $F\in\tau$, the set $\SE(F)$ belongs to $\mathcal B_{\tau}$; hence, the family covers $\SE(F)$. Let $B_1=\SE(H)$ and $B_2=\SE(K)$ be two members of $\mathcal B_{\tau}$. Lemma~\ref{lem:se-identities} gives
\[
B_1\cap B_2=\SE(H\softcap K).
\]
Since $H\softcap K\in\tau$, the intersection $B_1\cap B_2$ also belongs to $\mathcal B_{\tau}$. Therefore, $\mathcal B_{\tau}$ satisfies the basis axioms, and its arbitrary unions form the topology in \eqref{eq:e-topology}.
\end{proof}

\begin{remark}\label{rem:basis-not-topology}
The family $\mathcal B_{\tau}$ need not itself be a topology. Example~\ref{ex:strict-union} shows that
\[
\SE(H)\cup\SE(K)
\]
need not equal $\SE(H\softcup K)$. The arbitrary unions in \eqref{eq:e-topology} are therefore necessary.
\end{remark}

\begin{example}[Open coordinate images are not enough]\label{ex:projection-counterexample}
Let $A=\{1,2\}$, let $F(1)=F(2)=\{0,1\}$, and give both component sets the indiscrete topology. Put
\[
U=\{(0,0),(1,1)\},
\qquad
V=\{(0,0),(0,1),(1,0)\}.
\]
For both coordinates $t$, the images $\pi_t(U)$ and $\pi_t(V)$ equal $\{0,1\}$ and are open. However,
\[
U\cap V=\{(0,0)\},
\]
and both coordinate images of this intersection are singletons, which are not open. Thus, the subsets of $\SE(F)$ having open coordinate images do not generally form a topology. The basis in Theorem~\ref{thm:e-topology} avoids this problem.
\end{example}

\begin{example}[An induced open set need not be one basic set]\label{ex:not-one-basic-set}
Use the soft sets $H_0,H_1$ from Example~\ref{ex:strict-union} and let
\[
\tau=\{\Phi,H_0,H_1,F\}.
\]
Then,
\[
\SE(H_0)\cup\SE(H_1)=\{(0,0),(1,1)\}
\]
is open in $\tau^{\mathrm e}$. It is not equal to $\SE(H)$ for any soft set $H$. Indeed, both of its coordinate images are $\{0,1\}$. If it were $\SE(H)$, then $H(1)=H(2)=\{0,1\}$, which would give $\SE(H)=\SE(F)$.
\end{example}

\begin{proposition}\label{prop:canonical-product}
If $\tau=\Can(\{\sigma_t\}_{t\in A})$, then $\tau^{\mathrm e}$ is the box topology on
\[
\SE(F)=\prod_{t\in A}F(t)
\]
generated by the component topologies $\sigma_t$. If $A$ is finite, this is also the product topology.
\end{proposition}

\begin{proof}
A basic open set in the box topology has the form $\prod_{t\in A}U_t$, where $U_t\in\sigma_t$, for every $t\in A$. Define a soft set $H$ by $H(t)=U_t$. Then, $H\in\tau$ and
\[
\SE(H)=\prod_{t\in A}U_t.
\]
Conversely, every $H\in\tau$ satisfies $H(t)\in\sigma_t$, for every $t\in A$, so $\SE(H)$ is a basic box. Hence, the two bases agree. For finite $A$, the box and product topologies coincide.
\end{proof}

\begin{proposition}\label{prop:canonical-enlargement}
For an arbitrary soft topology $\tau$, define
\[
\tau^{\mathrm{can}}=\Can(\{\tau_t\}_{t\in A}).
\]
Then, $\tau\subseteq\tau^{\mathrm{can}}$, the two soft topologies have the same component topologies, and
\[
\tau^{\mathrm e}\subseteq(\tau^{\mathrm{can}})^{\mathrm e}.
\]
\end{proposition}

\begin{proof}
If $H\in\tau$, then $H(t)\in\tau_t$, for every $t\in A$. Hence, $H\in\tau^{\mathrm{can}}$, and $\tau\subseteq\tau^{\mathrm{can}}$. Proposition~\ref{prop:canonical-topology} shows that the component topology of $\tau^{\mathrm{can}}$ at $t$ is $\tau_t$. Finally, every basic set $\SE(H)$ with $H\in\tau$ is also a basic set for $(\tau^{\mathrm{can}})^{\mathrm e}$. Therefore, $\tau^{\mathrm e}\subseteq(\tau^{\mathrm{can}})^{\mathrm e}$.
\end{proof}

The last inclusion may be strict.

\begin{example}[A noncanonical soft topology]\label{ex:noncanonical}
Use $A$, $F$, $H_0$, and $H_1$ from Example~\ref{ex:strict-union}, and let
\[
\tau=\{\Phi,H_0,H_1,F\}.
\]
This is a soft topology because $H_0\softcap H_1=\Phi$ and $H_0\softcup H_1=F$. Both component topologies are discrete, but $\tau$ is not canonical. Its induced soft-element topology is
\[
\tau^{\mathrm e}
 =\{\varnothing,\{(0,0)\},\{(1,1)\},
     \{(0,0),(1,1)\},\SE(F)\}.
\]
The soft elements $(0,1)$ and $(1,0)$ have the same open neighborhoods, so $\tau^{\mathrm e}$ is not $T_0$. In contrast, $(\tau^{\mathrm{can}})^{\mathrm e}$ is discrete. Thus, the component topologies do not determine $\tau^{\mathrm e}$ without canonicality.
\end{example}

\subsection{Soft bitopological spaces}

\begin{definition}[Soft bitopological space \cite{SenelCagman2014,Ittanagi2014,NazShabirAli2015}]\label{def:soft-bitopology}
A \emph{soft bitopological space} is a triple $(F,\tau_1,\tau_2)$, where $\tau_1$ and $\tau_2$ are soft topologies on the same soft set $F$. It is \emph{canonical} if both $\tau_1$ and $\tau_2$ are canonical.
\end{definition}

\begin{theorem}\label{thm:induced-bitopology}
If $(F,\tau_1,\tau_2)$ is a soft bitopological space, then
\[
(\SE(F),\tau_1^{\mathrm e},\tau_2^{\mathrm e})
\]
is a bitopological space. We call it the \emph{induced soft-element bitopological space} of $(F,\tau_1,\tau_2)$.
\end{theorem}

\begin{proof}
By Theorem~\ref{thm:e-topology}, both $\tau_1^{\mathrm e}$ and $\tau_2^{\mathrm e}$ are topologies on the same set $\SE(F)$. This is exactly the definition of a bitopological space.
\end{proof}

% ====================================================================
\section{Pairwise soft separation axioms}\label{sec:separation}

We define pairwise soft separation by using soft elements and soft open sets.

\begin{definition}[Pairwise soft $T_0$]\label{def:soft-T0}
A soft bitopological space $(F,\tau_1,\tau_2)$ is \emph{pairwise soft $T_0$} if, for distinct $a,b\in\SE(F)$, either there is $H\in\tau_1$ such that
\[
a\softin H,\qquad b\softnotin H,
\]
or there is $K\in\tau_2$ such that
\[
b\softin K,\qquad a\softnotin K.
\]
\end{definition}

\begin{definition}[Pairwise soft $T_1$]\label{def:soft-T1}
A soft bitopological space $(F,\tau_1,\tau_2)$ is \emph{pairwise soft $T_1$} if, for distinct $a,b\in\SE(F)$, there are $H\in\tau_1$ and $K\in\tau_2$ such that
\[
a\softin H,\quad b\softnotin H,
\qquad
b\softin K,\quad a\softnotin K.
\]
\end{definition}

\begin{definition}[Pairwise soft $T_2$]\label{def:soft-T2}
A soft bitopological space $(F,\tau_1,\tau_2)$ is \emph{pairwise soft $T_2$} if, for distinct $a,b\in\SE(F)$, there are $H\in\tau_1$ and $K\in\tau_2$ such that
\[
a\softin H,\qquad b\softin K,
\qquad \SE(H)\cap\SE(K)=\varnothing.
\]
\end{definition}

The last condition says that no soft element belongs to both soft open sets. Requiring $H\softcap K=\Phi$ would be stronger because it would force $H(t)\cap K(t)=\varnothing$, for every $t\in A$. If distinct soft elements $a$ and $b$ agree at a parameter $t$, then any soft sets $H$ and $K$ with $a\softin H$ and $b\softin K$ satisfy $a(t)=b(t)\in H(t)\cap K(t)$. Thus, $H\softcap K=\Phi$ cannot separate that pair. For finite $A$,
\[
\SE(H)\cap\SE(K)=\varnothing
\]
is equivalent to $H(t)\cap K(t)=\varnothing$ for at least one $t\in A$.

\begin{proposition}\label{prop:hierarchy}
Pairwise soft $T_2$ implies pairwise soft $T_1$, and pairwise soft $T_1$ implies pairwise soft $T_0$.
\end{proposition}

\begin{proof}
Assume that $a\softin H$, $b\softin K$, and $\SE(H)\cap\SE(K)=\varnothing$. Since $b\in\SE(K)$, the equality of the intersection to the empty set gives $b\notin\SE(H)$; hence, $b\softnotin H$. Similarly, $a\softnotin K$. The two soft open sets therefore satisfy Definition~\ref{def:soft-T1}. The implication from pairwise soft $T_1$ to pairwise soft $T_0$ follows by retaining either of the two separating soft open sets.
\end{proof}

\subsection{Comparison with the induced soft-element bitopological space}

\begin{theorem}\label{thm:soft-induced-equivalence}
For $j\in\{0,1,2\}$, a soft bitopological space $(F,\tau_1,\tau_2)$ is pairwise soft $T_j$ if and only if its induced soft-element bitopological space
\[
(\SE(F),\tau_1^{\mathrm e},\tau_2^{\mathrm e})
\]
is pairwise $T_j$.
\end{theorem}

\begin{proof}
Suppose first that $(F,\tau_1,\tau_2)$ is pairwise soft $T_j$. Every set $\SE(H)$ with $H\in\tau_i$ is open in $\tau_i^{\mathrm e}$. Thus, the soft open sets in Definitions~\ref{def:soft-T0}--\ref{def:soft-T2} give the open sets required by Definition~\ref{def:classical-pairwise}. In the $T_2$ case, their induced basic sets are disjoint by definition. Hence, the induced soft-element bitopological space is pairwise $T_j$.

Conversely, suppose that the induced soft-element bitopological space is pairwise $T_j$. Let $a,b\in\SE(F)$ be distinct.

For $j=0$, assume first that some $U\in\tau_1^{\mathrm e}$ contains $a$ but not $b$. Since $\mathcal B_{\tau_1}$ is a basis, there is $H\in\tau_1$ such that
\[
a\in\SE(H)\subseteq U.
\]
Then, $a\softin H$ and $b\softnotin H$. If the separating open set belongs to $\tau_2^{\mathrm e}$, the same argument gives the second alternative in Definition~\ref{def:soft-T0}.

For $j=1$, apply the basis argument to both separating open sets. This gives $H\in\tau_1$ and $K\in\tau_2$ satisfying Definition~\ref{def:soft-T1}.

For $j=2$, choose disjoint open sets $U\in\tau_1^{\mathrm e}$ and $V\in\tau_2^{\mathrm e}$ with $a\in U$ and $b\in V$. Choose $H\in\tau_1$ and $K\in\tau_2$ such that
\[
a\in\SE(H)\subseteq U,
\qquad
b\in\SE(K)\subseteq V.
\]
Since $U\cap V=\varnothing$, we have $\SE(H)\cap\SE(K)=\varnothing$. Therefore, $(F,\tau_1,\tau_2)$ is pairwise soft $T_2$.
\end{proof}

\subsection{Comparison with component bitopological spaces}

For $i\in\{1,2\}$ and $t\in A$, let
\[
(\tau_i)_t=\{H(t):H\in\tau_i\}.
\]
By Proposition~\ref{prop:component-topology},
\[
(F(t),(\tau_1)_t,(\tau_2)_t)
\]
is a bitopological space. We call it the \emph{component bitopological space at $t$}.

\begin{theorem}\label{thm:soft-to-components}
Let $(F,\tau_1,\tau_2)$ be pairwise soft $T_j$, where $j\in\{0,1,2\}$. Every component bitopological space is pairwise $T_j$.
\end{theorem}

\begin{proof}
Fix $t_0\in A$ and distinct points $x,y\in F(t_0)$. Since $\SE(F)\neq\varnothing$, choose a soft element
$c\in\SE(F)$. Define $a,b\colon A\to X$ by
\[
a(t_0)=x,\qquad b(t_0)=y,
\qquad a(t)=b(t)=c(t)\quad\text{when }t\neq t_0.
\]
Both $a$ and $b$ belong to $\SE(F)$, and they are distinct.

For $j=0$, the pairwise soft $T_0$ condition gives one of two alternatives. Suppose that $H\in\tau_1$, $a\softin H$, and $b\softnotin H$. Since $a(t)=b(t)$ when $t\neq t_0$, failure of $b$ to belong to $H$ must occur at $t_0$. Thus,
\[
x=a(t_0)\in H(t_0),
\qquad
y=b(t_0)\notin H(t_0).
\]
Since $H(t_0)\in(\tau_1)_{t_0}$, it separates $x$ from $y$. The alternative involving $\tau_2$ is proved in the same way.

For $j=1$, choose $H\in\tau_1$ and $K\in\tau_2$ from Definition~\ref{def:soft-T1}. The preceding argument gives
\[
x\in H(t_0),\quad y\notin H(t_0),
\qquad
y\in K(t_0),\quad x\notin K(t_0).
\]
Hence, the component bitopological space is pairwise $T_1$.

For $j=2$, choose $H\in\tau_1$ and $K\in\tau_2$ such that $a\softin H$, $b\softin K$, and $\SE(H)\cap\SE(K)=\varnothing$. Suppose that some $z$ belonged to $H(t_0)\cap K(t_0)$. Define
\[
d(t_0)=z,
\qquad d(t)=c(t)\quad\text{when }t\neq t_0.
\]
For $t\neq t_0$, the common value $c(t)=a(t)=b(t)$ belongs to both $H(t)$ and $K(t)$. At $t_0$, the point $z$ also belongs to both sets. Hence, $d\in\SE(H)\cap\SE(K)$, a contradiction. Therefore,
\[
H(t_0)\cap K(t_0)=\varnothing,
\]
and the component bitopological space is pairwise $T_2$.
\end{proof}
For use in the next theorem, let $t_0\in A$ and
$U\subseteq F(t_0)$. We write $\Cyl_{t_0}(U)$ for the soft set
defined by
\[
\Cyl_{t_0}(U)(t)=
\begin{cases}
U, & t=t_0,\\
F(t), & t\neq t_0.
\end{cases}
\]
If $\tau$ is canonical and $U\in\tau_{t_0}$, then
$\Cyl_{t_0}(U)\in\tau$.

\begin{theorem}\label{thm:components-to-soft}
Let $(F,\tau_1,\tau_2)$ be canonical. If every component bitopological space is pairwise $T_j$, where $j\in\{0,1,2\}$, then $(F,\tau_1,\tau_2)$ is pairwise soft $T_j$.
\end{theorem}

\begin{proof}
Let $a,b\in\SE(F)$ be distinct. Choose $t_0\in A$ such that $a(t_0)\neq b(t_0)$.

For $j=0$, the component axiom gives either an open set $U\in(\tau_1)_{t_0}$ such that
\[
a(t_0)\in U,\qquad b(t_0)\notin U,
\]
or an open set $V\in(\tau_2)_{t_0}$ such that
\[
b(t_0)\in V,\qquad a(t_0)\notin V.
\]
In the first case, canonicality gives $\Cyl_{t_0}(U)\in\tau_1$, and this soft open set contains $a$ but not $b$. The second case gives the other alternative in Definition~\ref{def:soft-T0}.

For $j=1$, choose $U\in(\tau_1)_{t_0}$ and $V\in(\tau_2)_{t_0}$ such that
\[
a(t_0)\in U,\quad b(t_0)\notin U,
\qquad
b(t_0)\in V,\quad a(t_0)\notin V.
\]
The soft open sets $\Cyl_{t_0}(U)\in\tau_1$ and $\Cyl_{t_0}(V)\in\tau_2$ satisfy Definition~\ref{def:soft-T1}.

For $j=2$, choose $U\in(\tau_1)_{t_0}$ and $V\in(\tau_2)_{t_0}$ such that
\[
a(t_0)\in U,\qquad b(t_0)\in V,
\qquad U\cap V=\varnothing.
\]
Let $H=\Cyl_{t_0}(U)$ and $K=\Cyl_{t_0}(V)$. Canonicality gives $H\in\tau_1$ and $K\in\tau_2$. We have $a\softin H$ and $b\softin K$. No soft element can belong to both $H$ and $K$ because their values at $t_0$ are disjoint. Thus, $\SE(H)\cap\SE(K)=\varnothing$, and the space is pairwise soft $T_2$.
\end{proof}

\begin{corollary}\label{cor:three-equivalent}
Let $(F,\tau_1,\tau_2)$ be canonical and let $j\in\{0,1,2\}$. The following statements are equivalent:
\begin{enumerate}[label=(\roman*)]
\item $(F,\tau_1,\tau_2)$ is pairwise soft $T_j$;
\item the induced soft-element bitopological space $(\SE(F),\tau_1^{\mathrm e},\tau_2^{\mathrm e})$ is pairwise $T_j$;
\item the component bitopological space $(F(t),(\tau_1)_t,(\tau_2)_t)$ is pairwise $T_j$, for every $t\in A$.
\end{enumerate}
\end{corollary}

\begin{proof}
The equivalence of (i) and (ii) is Theorem~\ref{thm:soft-induced-equivalence}. Theorems~\ref{thm:soft-to-components} and~\ref{thm:components-to-soft} give the equivalence of (i) and (iii).
\end{proof}

\begin{example}[A canonical pairwise soft $T_2$ space]\label{ex:canonical-separation}
Let $A=\{p,q\}$, $F(p)=\{0,1\}$, and $F(q)=\{u,v\}$. Give both $F(p)$ and $F(q)$ the discrete topology, and let $\tau_1=\tau_2$ be the resulting canonical soft topology. Every component bitopological space is pairwise $T_2$. Hence, Corollary~\ref{cor:three-equivalent} shows that $(F,\tau_1,\tau_2)$ is pairwise soft $T_2$. For example, the soft elements $(0,u)$ and $(1,u)$ are separated by the soft open sets $\Cyl_p(\{0\})$ and $\Cyl_p(\{1\})$.
\end{example}

Canonicality is needed in Theorem~\ref{thm:components-to-soft}. In Example~\ref{ex:noncanonical}, take $\tau_1=\tau_2=\tau$. Every component bitopological space is discrete and therefore pairwise $T_2$, but the induced soft-element bitopological space is not even pairwise $T_0$. By Theorem~\ref{thm:soft-induced-equivalence}, the soft bitopological space is not pairwise soft $T_0$.

% ====================================================================
\section{Pairwise soft compactness}\label{sec:compactness}

\begin{definition}\label{def:soft-compact}
Let $(F,\tau_1,\tau_2)$ be a soft bitopological space and let $H\softsub F$. A family $\mathcal C\subseteq\tau_1\cup\tau_2$ is a \emph{pairwise soft open cover} of $H$ if
\begin{equation}\label{eq:soft-cover}
H(t)\subseteq\bigcup_{C\in\mathcal C}C(t),\qquad\text{for every }t\in A.
\end{equation}
The soft set $H$ is \emph{pairwise soft compact} if every pairwise soft open cover of $H$ has a finite subcover. The soft bitopological space is pairwise soft compact when $F$ has this property.
\end{definition}

The condition in \eqref{eq:soft-cover} is required at every parameter. It does not say that the ordinary sets $\SE(C)$ cover $\SE(H)$. Thus, pairwise soft compactness is different from pairwise compactness of the induced soft-element bitopological space. The next example makes the distinction explicit.

\begin{example}\label{ex:soft-cover-not-element-cover}
Let $A=\{1,2\}$, let $F(1)=F(2)=\{0,1\}$, and use the canonical discrete soft topology for both $\tau_1$ and $\tau_2$. Define
\[
C_0(t)=\{0\},\qquad C_1(t)=\{1\}
\]
for both parameters. The family $\{C_0,C_1\}$ is a pairwise soft open cover of $F$ because $C_0(t)\cup C_1(t)=F(t)$, for every $t\in A$. However,
\[
\SE(C_0)\cup\SE(C_1)=\{(0,0),(1,1)\},
\]
which does not cover $\SE(F)$. Therefore, a pairwise soft open cover need not be an ordinary open cover of the soft-element set.
\end{example}

\begin{theorem}\label{thm:soft-compact-to-component}
Let $(F,\tau_1,\tau_2)$ be canonical. If $H\softsub F$ is pairwise soft compact, then $H(t)$ is pairwise compact in the component bitopological space
\[
(F(t),(\tau_1)_t,(\tau_2)_t),\qquad\text{for every }t\in A.
\]
\end{theorem}

\begin{proof}
Fix $t_0\in A$. If $H(t_0)=\varnothing$, the conclusion is immediate. Let $\mathcal V$ be a pairwise open cover of $H(t_0)$. For $V\in\mathcal V$, set $i(V)=1$ if
$V\in(\tau_1)_{t_0}$, and set $i(V)=2$ otherwise. Then,
\[
V\in(\tau_{i(V)})_{t_0}.
\]
Put $K_V=\Cyl_{t_0}(V)$. Since the soft bitopological space is
canonical, $K_V\in\tau_{i(V)}$.

We claim that $\{K_V:V\in\mathcal V\}$ is a pairwise soft open cover of $H$. At $t_0$, this follows from
\[
H(t_0)\subseteq\bigcup_{V\in\mathcal V}V
 =\bigcup_{V\in\mathcal V}K_V(t_0).
\]
If $t\neq t_0$, then $K_V(t)=F(t)$, for every $V\in\mathcal V$, so the cover condition also holds at $t$. Pairwise soft compactness gives $V_1,\ldots,V_n\in\mathcal V$ such that
\[
H\softsub K_{V_1}\softcup\cdots\softcup K_{V_n}.
\]
Taking the value at $t_0$ gives
\[
H(t_0)\subseteq V_1\cup\cdots\cup V_n.
\]
Hence, $H(t_0)$ is pairwise compact.
\end{proof}

\begin{theorem}\label{thm:finite-converse}
Assume that $A$ is finite. Let $(F,\tau_1,\tau_2)$ be a soft bitopological space and let $H\softsub F$. If $H(t)$ is pairwise compact in the component bitopological space
\[
(F(t),(\tau_1)_t,(\tau_2)_t),\qquad\text{for every }t\in A,
\]
then $H$ is pairwise soft compact.
\end{theorem}

\begin{proof}
Let $\mathcal C$ be a pairwise soft open cover of $H$. For a fixed $t\in A$, the family
\[
\{C(t):C\in\mathcal C\}
\]
is a pairwise open cover of $H(t)$. By the pairwise compactness of $H(t)$, there is a finite subfamily $\mathcal C_t\subseteq\mathcal C$ such that
\[
H(t)\subseteq\bigcup_{C\in\mathcal C_t}C(t).
\]
Since $A$ is finite, the family
\[
\mathcal C_0=\bigcup_{t\in A}\mathcal C_t
\]
is finite. It satisfies \eqref{eq:soft-cover} at every parameter and is therefore a finite pairwise soft subcover of $H$.
\end{proof}

\begin{corollary}\label{cor:compact-equivalence}
Let $A$ be finite and let $(F,\tau_1,\tau_2)$ be canonical. A soft subset $H\softsub F$ is pairwise soft compact if and only if $H(t)$ is pairwise compact in the component bitopological space, for every $t\in A$.
\end{corollary}

\begin{example}[The finite-parameter assumption is needed]\label{ex:infinite-compact}
Let $A$ be infinite and let $F(t)=\{0,1\}$, for every $t\in A$. Take $\tau_1=\tau_2$ to be the canonical soft topology generated by the discrete topology on each $F(t)$. Every soft subset of $F$ is open, and every component bitopological space is finite and pairwise compact.

For every $t\in A$, define $H_t\softsub F$ by
\[
H_t(t)=\{1\},
\qquad
H_t(s)=\{0\}\quad\text{when }s\neq t.
\]
Fix $s\in A$. Since $A$ is infinite, the family contains $H_s$, whose value at $s$ is $\{1\}$, and also some $H_t$ with $t\neq s$, whose value at $s$ is $\{0\}$. Hence,
\[
\bigcup_{t\in A}H_t(s)=\{0,1\}=F(s).
\]
Thus, $\{H_t:t\in A\}$ is a pairwise soft open cover of $F$.

Now take a finite subfamily $H_{t_1},\ldots,H_{t_n}$. Choose
$s\in A\setminus\{t_1,\ldots,t_n\}$. Then,
\[
\bigcup_{k=1}^{n}H_{t_k}(s)=\{0\}\neq F(s).
\]
No finite subfamily covers $F$. Therefore, componentwise pairwise compactness does not imply pairwise soft compactness for an infinite parameter set.
\end{example}

% ====================================================================
\section{Conclusion}

A soft topology $\tau$ on $F$ determines an ordinary topology on $\SE(F)$ once the sets $\SE(H)$, $H\in\tau$, are used as a basis. This distinction matters because the basic sets need not be closed under unions. With this construction, two soft topologies give the induced soft-element bitopological space $(\SE(F),\tau_1^{\mathrm e},\tau_2^{\mathrm e})$.

Pairwise soft $T_0$, $T_1$, and $T_2$ agree with the corresponding pairwise axioms of this induced space. In the canonical case, they also agree with the same axioms in every component bitopological space. Pairwise soft compactness passes from a canonical soft bitopological space to its components. The converse holds for finitely many parameters, but Example~\ref{ex:infinite-compact} shows that component compactness alone is not enough when the parameter set is infinite.

One open problem is to find useful conditions, weaker than finiteness of $A$, under which componentwise pairwise compactness still implies pairwise soft compactness. Other questions concern pairwise soft regularity, normality, and continuous mappings, and how these notions are reflected by the induced soft-element bitopological space.

% ====================================================================
\section*{Author Contributions}
S. Ray conceived and supervised the study. S. Ray and T. Barik checked the arguments, prepared the manuscript, and approved the final version.

\section*{Conflicts of Interest}
The authors declare no conflict of interest.

\section*{Ethical Review and Approval}
No ethical approval was required for this study.

% ====================================================================


\begin{thebibliography}{99}

\bibitem{Molodtsov1999}
D.~Molodtsov,
\newblock Soft set theory---first results,
\newblock \emph{Computers \& Mathematics with Applications} \textbf{37} (1999), 19--31,
\newblock doi: 10.1016/S0898-1221(99)00056-5.

\bibitem{Maji2003}
P.~K. Maji, R.~Biswas, and A.~R. Roy,
\newblock Soft set theory,
\newblock \emph{Computers \& Mathematics with Applications} \textbf{45} (2003), 555--562,
\newblock doi: 10.1016/S0898-1221(03)00016-6.

\bibitem{PeiMiao2005}
D.~Pei and D.~Miao,
\newblock From soft sets to information systems,
\newblock in: \emph{Proceedings of the 2005 IEEE International Conference on Granular Computing}, Vol.~2, IEEE, 2005, pp.~617--621,
\newblock doi: 10.1109/GRC.2005.1547365.

\bibitem{Ali2009}
M.~I. Ali, F.~Feng, X.~Liu, W.~K. Min, and M.~Shabir,
\newblock On some new operations in soft set theory,
\newblock \emph{Computers \& Mathematics with Applications} \textbf{57}(9) (2009), 1547--1553,
\newblock doi: 10.1016/j.camwa.2008.11.009.

\bibitem{CagmanEnginoglu2010}
N.~\c{C}a\u{g}man and S.~Engino\u{g}lu,
\newblock Soft set theory and uni-int decision making,
\newblock \emph{European Journal of Operational Research} \textbf{207}(2) (2010), 848--855,
\newblock doi: 10.1016/j.ejor.2010.05.004.

\bibitem{Cagman2011}
N.~\c{C}a\u{g}man, S.~Karata\c{s}, and S.~Engino\u{g}lu,
\newblock Soft topology,
\newblock \emph{Computers \& Mathematics with Applications} \textbf{62} (2011), 351--358,
\newblock doi: 10.1016/j.camwa.2011.05.016.

\bibitem{ShabirNaz2011}
M.~Shabir and M.~Naz,
\newblock On soft topological spaces,
\newblock \emph{Computers \& Mathematics with Applications} \textbf{61} (2011), 1786--1799,
\newblock doi: 10.1016/j.camwa.2011.02.006.

\bibitem{Zorlutuna2012}
\.I.~Zorlutuna, M.~Akda\u{g}, W.~K. Min, and S.~Atmaca,
\newblock Remarks on soft topological spaces,
\newblock \emph{Annals of Fuzzy Mathematics and Informatics} \textbf{3}(2) (2012), 171--185.

\bibitem{Peyghan2014}
E.~Peyghan, B.~Samadi, and A.~Tayebi,
\newblock Some results related to soft topological spaces,
\newblock \emph{Facta Universitatis, Series: Mathematics and Informatics} \textbf{29}(4) (2014), 325--336.

\bibitem{HussainAhmad2015}
S.~Hussain and B.~Ahmad,
\newblock Soft separation axioms in soft topological spaces,
\newblock \emph{Hacettepe Journal of Mathematics and Statistics} \textbf{44}(3) (2015), 559--568,
\newblock doi: 10.15672/HJMS.2015449426.

\bibitem{BayramovAras2018}
S.~Bayramov and C.~G. Aras,
\newblock A new approach to separability and compactness in soft topological spaces,
\newblock \emph{TWMS Journal of Pure and Applied Mathematics} \textbf{9}(1) (2018), 82--93.

\bibitem{RayGoldar2017}
S.~Ray and S.~Goldar,
\newblock Soft set and soft group from classical view point,
\newblock \emph{Journal of the Indian Mathematical Society} \textbf{84}(3--4) (2017), 273--286,
\newblock doi: 10.18311/jims/2017/6123.

\bibitem{ChineySamanta2019}
M.~Chiney and S.~K. Samanta,
\newblock Soft topology redefined,
\newblock \emph{The Journal of Fuzzy Mathematics} \textbf{27}(2) (2019), 459--486.

\bibitem{GoldarRay2019}
S.~Goldar and S.~Ray,
\newblock A study of soft topological axioms and soft compactness by using soft elements,
\newblock \emph{Journal of New Results in Science} \textbf{8}(2) (2019), 53--66.

\bibitem{NazmulSamanta2013}
S.~Nazmul and S.~K. Samanta,
\newblock Neighbourhood properties of soft topological spaces,
\newblock \emph{Annals of Fuzzy Mathematics and Informatics} \textbf{6}(1) (2013), 1--15.

\bibitem{Matejdes2016}
M.~Matejdes,
\newblock Soft topological space and topology on the Cartesian product,
\newblock \emph{Hacettepe Journal of Mathematics and Statistics} \textbf{45}(4) (2016), 1091--1100,
\newblock doi: 10.15672/HJMS.20164513117.

\bibitem{Matejdes2021}
M.~Matejdes,
\newblock Methodological remarks on soft topology,
\newblock \emph{Soft Computing} \textbf{25}(5) (2021), 4149--4156,
\newblock doi: 10.1007/s00500-021-05587-7.

\bibitem{TaskopruAltintas2021}
K.~Ta\c{s}k\"opr\"u and \.{I}.~Alt\i nta\c{s},
\newblock A new approach for soft topology and soft function via soft element,
\newblock \emph{Mathematical Methods in the Applied Sciences} \textbf{44}(9) (2021), 7556--7570,
\newblock doi: 10.1002/mma.6354.

\bibitem{AltintasTaskopruSelvi2021}
\.{I}.~Alt\i nta\c{s}, K.~Ta\c{s}k\"opr\"u, and B.~Selvi,
\newblock Countable and separable elementary soft topological space,
\newblock \emph{Mathematical Methods in the Applied Sciences} \textbf{44}(9) (2021), 7811--7819,
\newblock doi: 10.1002/mma.6976.

\bibitem{Azzam2022}
A.~A. Azzam, Z.~A. Ameen, T.~M. Al-shami, and M.~E. El-Shafei,
\newblock Generating soft topologies via soft set operators,
\newblock \emph{Symmetry} \textbf{14}(5) (2022), Article~914,
\newblock doi: 10.3390/sym14050914.

\bibitem{Kelly1963}
J.~C. Kelly,
\newblock Bitopological spaces,
\newblock \emph{Proceedings of the London Mathematical Society} \textbf{13} (1963), 71--89,
\newblock doi: 10.1112/plms/s3-13.1.71.

\bibitem{FletcherHoylePatty1969}
P.~Fletcher, H.~B. Hoyle III, and C.~W. Patty,
\newblock The comparison of topologies,
\newblock \emph{Duke Mathematical Journal} \textbf{36} (1969),
325--331,
\newblock doi: 10.1215/S0012-7094-69-03641-2.

\bibitem{Reilly1972}
I.~L. Reilly,
\newblock On bitopological separation properties,
\newblock \emph{Nanta Mathematica} \textbf{5} (1972), 14--25.

\bibitem{CookeReilly1975}
I.~E. Cooke and I.~L. Reilly,
\newblock On bitopological compactness,
\newblock \emph{Journal of the London Mathematical Society} (2) \textbf{9}(4) (1975), 518--522,
\newblock doi: 10.1112/jlms/s2-9.4.518.

\bibitem{SenelCagman2014}
G.~\c{S}enel and N.~\c{C}a\u{g}man,
\newblock Soft closed sets on soft bitopological space,
\newblock \emph{Journal of New Results in Science} \textbf{3}(5) (2014), 57--66.

\bibitem{Ittanagi2014}
B.~M. Ittanagi,
\newblock Soft bitopological spaces,
\newblock \emph{International Journal of Computer Applications} \textbf{107}(7) (2014), 1--4,
\newblock doi: 10.5120/18760-0038.

\bibitem{NazShabirAli2015}
M.~Naz, M.~Shabir, and M.~I. Ali,
\newblock Separation axioms in bi-soft topological spaces,
\newblock arXiv:1509.00866 (2015),
\newblock doi: 10.48550/arXiv.1509.00866.

\bibitem{AcharjeeTripathy2017}
S.~Acharjee and B.~C. Tripathy,
\newblock Some results on soft bitopology,
\newblock \emph{Boletim da Sociedade Paranaense de Matem\'atica} \textbf{35}(1) (2017), 269--279,
\newblock doi: 10.5269/bspm.v35i1.29145.

\bibitem{Nordo2022}
G.~Nordo,
\newblock Soft $N$-topological spaces,
\newblock \emph{Poincar\'e Journal of Analysis and Applications} \textbf{9}(2) (2022), 249--273,
\newblock doi: 10.46753/pjaa.2022.v09i02.009.

\bibitem{Askandar2022}
S.~Askandar, B.~S. Abdullah, and L.~Khaleel,
\newblock Soft $i$-open sets in soft bi-topological spaces,
\newblock \emph{Al-Qadisiyah Journal of Pure Science} \textbf{27}(1) (2022), 124--136,
\newblock doi: 10.29350/qjps.2022.27.1.1592.

\bibitem{ReillyLocal1972}
I.~L. Reilly,
\newblock Bitopological local compactness,
\newblock \emph{Indagationes Mathematicae (Proceedings)}
\textbf{75}(5) (1972), 407--411,
\newblock doi: 10.1016/1385-7258(72)90037-6.

\end{thebibliography}
\end{document}